\documentclass[a4paper,12pt,twoside,leqno,final]{amsart}
\usepackage{amsmath}
\usepackage{amssymb}

\newtheorem{thm}{Theorem}[section]
\newtheorem{lem}[thm]{Lemma}

\newtheorem{prop}[thm]{Proposition}
\newtheorem{cor}[thm]{Corollary}

\newtheorem*{tha}{Theorem A}
\newtheorem*{thb}{Theorem B}

\newcommand{\C}{{\mathbb C}}
\newcommand{\D}{{\mathbb D}}
\newcommand{\R}{{\mathbb R}}
\newcommand{\T}{{\mathbb T}}

\newcommand{\N}{{\mathbb N}}

\newcommand{\bmo}{{\rm BMO}}
\newcommand{\bmoa}{{\rm BMOA}}

\newcommand{\Al}{A^\alpha}

\renewcommand{\sb}{\subset}

\newcommand{\eps}{\varepsilon}

\newcommand{\f}{\frac}

\newcommand{\al}{\alpha}
\newcommand{\be}{\beta}
\newcommand{\Ga}{\Gamma}
\newcommand{\ga}{\gamma}

\newcommand{\de}{\delta}
\newcommand{\la}{\lambda}
\newcommand{\ze}{\zeta}
\renewcommand{\th}{\theta}
\newcommand{\si}{\sigma}

\newcommand{\Om}{\Omega}
\newcommand{\Omte}{\Om(\th,\eps)}

\numberwithin{equation}{section}

\title[Wronskians and deep zeros of holomorphic functions]
{Wronskians and deep zeros\\ 
of holomorphic functions}
\author{Konstantin M. Dyakonov}
\address{ICREA and Universitat de Barcelona, Departament de Matem\`atica 
Aplicada i An\`alisi, Gran Via 585, E-08007 Barcelona, Spain}
\email{konstantin.dyakonov@icrea.cat}
\thanks{Supported in part by grant MTM2011-27932-C02-01 from El Ministerio de Ciencia 
e Innovaci\'on (Spain) and grant 2009-SGR-1303 from AGAUR (Generalitat de Catalunya).}

\begin{document}
\begin{abstract}
Given linearly independent holomorphic functions $f_0,\dots,f_n$ on a planar domain $\Om$, 
let $\mathcal E$ be the set of those points $z\in\Om$ where a nontrivial linear combination 
$\sum_{j=0}^n\la_jf_j$ may have a zero of multiplicity greater than $n$, once the 
coefficients $\la_j=\la_j(z)$ are chosen appropriately. An elementary argument 
involving the Wronskian $W$ of the $f_j$'s shows that $\mathcal E$ is a discrete subset 
of $\Om$ (and is actually the zero set of $W$); thus \lq\lq deep" zeros are rare. We 
elaborate on this by studying similar phenomena in various function spaces on the unit disk, 
with more sophisticated boundary smallness conditions playing the role of deep zeros. 

\bigskip

\smallskip

{\bf R\'esum\'e.} Etant donn\'ees des fonctions holomorphes $f_0,\dots,f_n$ 
lin\'eairement ind\'ependantes sur un domaine $\Om$ du plan, soit $\mathcal E$ l'ensemble 
des points $z\in\Om$ o\`u une combinaison lin\'eaire non triviale $\sum_{j=0}^n\la_jf_j$ 
peut avoir un z\'ero d'ordre sup\'erieur \`a $n$. Un argument \'el\'ementaire utilisant 
le wronskien des $f_j$ montre que $\mathcal E$ est un sous-ensemble discret de $\Om$; 
ainsi, les z\'eros \lq\lq profonds" sont rares. Nous \'etudions des ph\'enom\`enes 
similaires dans divers espaces de fonctions sur le disque unit\'e, avec des conditions 
plus sophistiqu\'ees de d\'ecroissance au bord \`a la place de z\'eros profonds 
int\'erieurs. 
\end{abstract}

\maketitle

\section{Introduction} 

Let $\Om$ be a domain in the complex plane $\C$, and let $\mathcal H(\Om)$ denote the set 
of all holomorphic functions on $\Om$. The classical uniqueness theorem tells us that, 
given a non-null function $f\in\mathcal H(\Om)$, the zero set $\mathcal Z(f):=\{z\in\Om:f(z)=0\}$ 
is discrete in $\Om$ (i.\,e., has no accumulation points therein). Less known is the fact that 
this admits a natural extension to linear combinations of several functions, provided that 
the zeros are required to have suitably high multiplicities. Before stating the result, we need 
to introduce a bit of terminology. Namely, given a nonnegative integer $n$, a function 
$f\in\mathcal H(\Om)$ and a point $z_0\in\mathcal Z(f)$, we say that the zero $z_0$ is {\it $n$-deep} 
for $f$ if its multiplicity is at least $n+1$. 

\begin{tha} Suppose $f_0,\dots,f_n$ are linearly independent holomorphic functions on a domain 
$\Om\subset\C$. Then there is a discrete subset $\mathcal E$ of $\Om$ with the following property: 
whenever $\la_0,\dots,\la_n$ are complex numbers with $\sum_{j=0}^n|\la_j|>0$, the $n$-deep zeros of 
the function $\sum_{j=0}^n\la_jf_j$ are all contained in $\mathcal E$. 
\end{tha}

Thus, $n$-deep zeros are forbidden for a nontrivial linear combination $\sum_{j=0}^n\la_jf_j$ 
except on a \lq\lq thin" set, which depends only on the $f_j$'s but not on the $\la_j$'s. Of course, 
this is no longer true with \lq\lq$(n-1)$-deep" in place of \lq\lq$n$-deep"; to see why, consider 
the polynomials $(z-a)^n$ with $a\in\Om$. 

\par We strongly believe that Theorem A should be known. However, having found no reference for it 
in the literature, we give a simple proof instead. 

\medskip\noindent{\it Proof of Theorem A.} For a point $z\in\Om$ to be an $n$-deep zero of 
$g:=\sum_{j=0}^n\la_jf_j$, it is necessary and sufficient that 
$$g(z)=g'(z)=\dots=g^{(n)}(z)=0.$$ 
We now rewrite this as 
\begin{equation}\label{eqn:sys}
\sum_{j=0}^n\la_jf_j^{(k)}(z)=0\qquad(k=0,\dots,n)
\end{equation}
and view \eqref{eqn:sys} as a system of homogeneous linear equations with \lq\lq unknowns" $\la_j$. 
A nontrivial solution $(\la_0,\dots,\la_n)$ to \eqref{eqn:sys} will therefore exist if and only if the 
coefficient matrix 
$$\left\{f_j^{(k)}(z):\,j,k=0,\dots,n\right\}$$ 
is singular. In other words, the {\it Wronskian} $W=W(f_0,\dots,f_n)$ defined by 
\begin{equation}\label{eqn:wro}
W(f_0,\dots,f_n):=
\begin{vmatrix}
f_0&f_1&\dots&f_n\\
f'_0&f'_1&\dots&f'_n\\
\dots&\dots&\dots&\dots\\
f_0^{(n)}&f_1^{(n)}&\dots&f_n^{(n)}
\end{vmatrix}
\end{equation}
must vanish at $z$. (We mention in passing that, according to some authors, the credit 
for introducing determinants of the form \eqref{eqn:wro} should definitely be shared 
by Wronski with Froufrou.) The $f_j$'s being linearly independent, 
it follows that $W\not\equiv0$; see \cite[Chapter 1]{La}. 
Of course, it is also true that $W\in\mathcal H(\Om)$, so 
the zero set $\mathcal Z(W)=:\mathcal E$ is a discrete subset of $\Om$. On the other hand, 
we have just seen that $\mathcal E$ consists of precisely those points in $\Om$ which 
can be realized as $n$-deep zeros for nontrivial linear combinations of the $f_j$'s. 
\quad\qed

\medskip 
The proof tells us that the $n$-deep zeros of {\it all} the linear combinations as above 
coincide with the zeros of a {\it single} holomorphic function, namely, of $W$. In some 
special cases, one is able to compute $W$ explicitly and then to determine the exceptional set 
$\mathcal E=\mathcal Z(W)$. We take the liberty of including one such result, which concerns 
the zeros of a \lq\lq fewnomial" (i.e., a possibly lacunary polynomial) of the form 
\begin{equation}\label{eqn:few}
P(z)=\sum_{j=0}^na_jz^{d_j}, 
\end{equation}
as well as those of an exponential sum 
\begin{equation}\label{eqn:expsum}
Q(z)=\sum_{j=0}^na_je^{\mu_jz}.
\end{equation}

\begin{cor}\label{cor:fewnomial} Let $a_0,\dots,a_n$ be complex numbers with $\sum_{j=0}^n|a_j|>0$. 
\par{\rm (a)} Given nonnegative integers $d_0<d_1<\dots<d_n$, the polynomial $P$ defined by \eqref{eqn:few} 
has no $n$-deep zeros in $\C\setminus\{0\}$. 
\par{\rm (b)} Given pairwise distinct complex numbers $\mu_0,\dots,\mu_n$, the function $Q$ defined 
by \eqref{eqn:expsum} has no $n$-deep zeros in $\C$. 
\end{cor}

To prove (a), one notes that the Wronskian $W=W(z^{d_0},\dots,z^{d_n})$ is a monomial. Indeed, it equals 
$cz^d$ with suitable integers $c\ne0$ and $d\ge0$ depending on the $d_j$'s and on $n$ (a precise formula 
can be found in \cite{BD}). Thus $W$ has no zeros in $\C\setminus\{0\}$, and accordingly, $P$ has no $n$-deep 
zeros except possibly at $0$. 
\par To prove (b), one checks that the Wronskian $W(e^{\mu_0z},\dots,e^{\mu_nz})$ 
is a constant multiple of $e^{\mu z}$, where $\mu=\sum_{j=0}^n\mu_j$, the constant factor being 
nonzero. This time, we see that the Wronskian is nowhere zero, and the required fact follows. 
\par One may observe that part (a) is actually a consequence of (b), and anyhow, both statements 
are probably -- if not certainly -- known. Nevertheless, just as with Theorem A above, we have 
found it easier to give a quick proof than to search for a reference. 

\par While Theorem A is essentially \lq\lq algebraic" in nature, we are interested in extending it to a more 
\lq\lq analytic" context. In what follows, the domain $\Om$ is (almost always) taken to be the disk 
$\D:=\{z\in\C:|z|<1\}$, the functions $f_0,\dots,f_n$ are assumed to lie in a certain space $X\subset\mathcal H(\D)$, 
and the smallness condition imposed on the linear combination $\sum_{j=0}^n\la_jf_j$ involves some sort of 
decay near (some parts of) the circle $\T:=\partial\D$ rather than having deep zeros inside. Once the class 
$X$ and the decay condition are chosen appropriately (the latter being sufficiently strong), the phenomenon 
underlying Theorem A will manifest itself in some form or other, and we find various instances of this. 

\par Typically, the union of the smallness sets that correspond to {\it all} the nontrivial 
linear combinations of $f_0,\dots,f_n(\in X)$ turns out to be \lq\lq thin", and can be realized 
as a set on which a {\it single} nontrivial function (from a certain space $\widetilde X\subset\mathcal H(\D)$ 
related to $X$) is small, possibly not in the original sense. This general principle does not seem to have 
been either noticed or reflected in the literature, so we wish to highlight it here. Furthermore, a number 
of concrete quantitative statements will be supplied to illustrate it. Wronskians and their basic properties 
will again play an appreciable role in the proofs, but more sophisticated tools from complex analysis 
will also be needed. 

\par In Section 2 below, we deal with \lq\lq large analytic functions" on $\D$ (for which 
a certain controlled growth near $\T$ is allowed) and study their nontangential decay near $\T$. 
In Sections 3 and 4, we turn to spaces of \lq\lq smooth analytic functions" (this time, 
a boundary smoothness condition is imposed) and look at the inner factors of such functions. 
Finally, in Section 5, we consider several types of holomorphic spaces $X$ and discuss the exceptional 
sets $\mathcal E$ that arise in Theorem A when the functions $f_j$ range over $X$. 

\par In conclusion, we mention that this paper shares some common features with the author's recent 
work in \cite{DCRM, DMA, DCONM}, where Wronskians were employed in connection with function-theoretic 
analogues of the so-called $abc$ conjecture. Also, a portion of our current Section 4 was previously 
announced in \cite{DCRM12}, in rather a sketchy form. 

\section{Large analytic functions that are small near the boundary} 

This section is devoted to the {\it Korenblum classes} $A^{-\be}$ with $\be>0$; 
here $A^{-\be}$ is defined as the set of all functions $f\in\mathcal H(\D)$ that satisfy 
\begin{equation}\label{eqn:koren}
\sup_{z\in\D}|f(z)|(1-|z|)^\be<\infty.
\end{equation}
A discussion of these spaces, as well as of $A^{-\infty}:=\bigcup_{\be>0}A^{-\be}$, 
can be found in \cite{K}. 
\par More generally, given a number $\al\in\R$, we denote by $\Al$ the set of 
those $f\in\mathcal H(\D)$ for which 
\begin{equation}\label{eqn:lipkor}
\sup_{z\in\D}|f^{(m)}(z)|(1-|z|)^{m-\al}<\infty,
\end{equation}
where $m$ is the least nonnegative integer in the interval $(\al,\infty)$. When $\al=-\be<0$, one 
takes $m=0$ and recovers the growth condition \eqref{eqn:koren}. When $\al>0$, \eqref{eqn:lipkor} 
becomes a smoothness condition on $\T$ that characterizes the classical {\it Lipschitz--Zygmund spaces} 
(to be dealt with in the next section). Finally, the value $\al=0$ corresponds to the {\it Bloch 
space} $A^0$, which is usually denoted by $\mathcal B$; see \cite{ACP}. 

\par Now, for a point $\ze\in\T$ and a number $M>1$, we write 
$$\Ga_M(\ze):=\{z\in\D:\,|\ze-z|\le M(1-|z|)\};$$ 
thus $\Ga_M(\ze)$ is a {\it Stolz angle} (or {\it cone}) with vertex $\ze$. 
Further, given a number $\ga>0$, we say that a function $f\in\mathcal H(\D)$ is {\it nontangentially 
small of order $\ga$ at $\ze$} if 
$$f(z)=O\left((1-|z|)^\ga\right)\quad\text{as}\quad|z|\to1,\,\,z\in \Ga_M(\ze),$$ 
for some $M>1$. Accordingly, by saying that $f$ is {\it nontangentially small of order $>\ga$ at $\ze$} 
we mean that 
$$f(z)=o\left((1-|z|)^\ga\right)\quad\text{as}\quad|z|\to1,\,\,z\in \Ga_M(\ze),$$ 
for some $M>1$. Finally, if $f_0,\dots,f_n$ are linearly independent functions in $\mathcal H(\D)$, 
then we denote by $E_\ga(f_0,\dots,f_n)$ the set of all points $\ze\in\T$ with the following 
property: there exists a nontrivial linear combination $\sum_{j=0}^n\la_jf_j$ which is nontangentially 
small of order $>\ga$ at $\ze$ (the coefficients $\la_j$ will of course depend on $\ze$). 

\begin{thm}\label{thm:kornontang} Let $f_0,\dots,f_n$ be linearly independent functions in $A^{-\be}$ 
(with $\be>0$), and put 
\begin{equation}\label{eqn:defga}
\ga=n\be+\f{n(n+1)}2.
\end{equation}
Then $E_\ga(f_0,\dots,f_n)$ is a set of Lebesgue measure $0$ on $\T$. The same is true for $\be=0$, provided 
that $A^{-\be}$ is replaced by $H^\infty$, the space of bounded analytic functions. 
\end{thm}

We do not know whether the value given by \eqref{eqn:defga} is optimal. Note, however, that the result 
breaks down for all $\ga<n$. Indeed, for any fixed $\ze\in\T$, the function $z\mapsto(z-\ze)^n$ 
is a linear combination of $1,z,\dots,z^n$ and is nontangentially small of order $n$ at $\ze$. 
Thus, $E_\ga(1,z,\dots,z^n)=\T$ whenever $\ga<n$. This shows, in particular, that \eqref{eqn:defga} 
{\it is} optimal for $\be=0$ and $n=1$. 

\par Before proceeding with the proof of Theorem \ref{thm:kornontang}, we state and prove a preliminary 
result to rely upon. Below, we write $\N$ for the set of positive integers, and we use the 
notation $B(z,r)$ for the disk $\{w:|w-z|<r\}$. 

\begin{lem}\label{lem:decay} Let $\al\in\R$, $\de\in(0,1)$ and $k\in\N$. 
Further, let $G$ and $G_0$ be subsets of $\D$ such that 
\begin{equation}\label{eqn:ggg}
\bigcup_{z\in G}B(z,\de(1-|z|))\subset G_0.
\end{equation}
Then, for every function $f\in\mathcal H(\D)$ satisfying 
\begin{equation}\label{eqn:estfg}
f(z)=O\left((1-|z|)^\al\right),\qquad z\in G_0,
\end{equation}
we have 
\begin{equation}\label{eqn:estderf}
f^{(k)}(z)=O\left((1-|z|)^{\al-k}\right),\qquad z\in G.
\end{equation}
The constant in the latter $O$-condition depends on that in \eqref{eqn:estfg}, as well as 
on $\al$, $\de$ and $k$. 
\end{lem}

\begin{proof} Fix $z\in G$ and consider the circle 
$$\ga_z=\ga_{z,\de}:=\{\ze\in\C:|\ze-z|=\de(1-|z|)\}.$$ 
We have then 
$$f^{(k)}(z)=\f{k!}{2\pi i}\int_{\ga_z}\f{f(\ze)}{(\ze-z)^{k+1}}d\ze,$$ 
whence 
\begin{equation}\label{eqn:chainik}
|f^{(k)}(z)|\le\f{k!}{\de^k(1-|z|)^k}\cdot\sup\{|f(\ze)|:\,\ze\in\ga_z\}.
\end{equation}
It follows from \eqref{eqn:ggg} that $\ga_z\subset\text{\rm clos}\,G_0$, and so  
\eqref{eqn:estfg} yields 
$$|f(\ze)|\le C(1-|\ze|)^\al,\qquad\ze\in\ga_z,$$
with a constant $C>0$. We now combine this with the inequalities 
$$1-\de\le\f{1-|\ze|}{1-|z|}\le1+\de,\qquad\ze\in\ga_z,$$
to get 
$$\sup\{|f(\ze)|:\,\ze\in\ga_z\}\le\widetilde C(1-|z|)^\al,\qquad z\in G,$$
with a suitable $\widetilde C=\widetilde C(\al,\de,C)$. Plugging this last estimate 
into \eqref{eqn:chainik}, we arrive at \eqref{eqn:estderf}. 
\end{proof}

\par We also need the \lq\lq little oh" version of the above lemma, which can be established 
in a similar way.  

\begin{lem}\label{lem:little} Assume, under the hypotheses of Lemma \ref{lem:decay}, that 
$\T\cap\text{\rm clos}\,G\ne\emptyset$. Then, given a function $f\in\mathcal H(\D)$ satisfying 
$$f(z)=o\left((1-|z|)^\al\right)\quad\text{as}\quad|z|\to1,\,\,z\in G_0,$$
it follows that 
$$f^{(k)}(z)=o\left((1-|z|)^{\al-k}\right)\quad\text{as}\quad|z|\to1,\,\,z\in G.$$
\end{lem}

\medskip\noindent{\it Proof of Theorem \ref{thm:kornontang}.} Let $\ze\in E_\ga(f_0,\dots,f_n)$, 
so that there exist coefficients $\la_j=\la_j(\ze)$ with $\sum_{j=0}^n|\la_j|>0$ which make the 
linear combination 
\begin{equation}\label{eqn:licog}
\sum_{j=0}^n\la_jf_j=:g,
\end{equation}
nontangentially small of order $>\ga$ at $\ze$. At least one of the coefficients, say $\la_k$, is thus nonzero. 
Recalling the notation \eqref{eqn:wro} for the Wronskian, we put 
$$W:=W(f_0,\dots,f_n)$$ 
and 
$$W_k:=W(f_0,\dots,f_{k-1},g,f_{k+1},\dots,f_n).$$ 
It should be noted that $W\not\equiv0$, because the $f_j$'s are linearly independent holomorphic functions.  
Furthermore, it follows from \eqref{eqn:licog} that $W_k=\la_kW$. 
\par Expanding the determinant $W_k$ along its $k$th column (i.\,e., the one that contains $g,g',\dots,g^{(n)}$), 
we get 
\begin{equation}\label{eqn:wrok}
W_k=\sum_{l=0}^ng^{(l)}\Delta_l,
\end{equation}
where $\Delta_l=\Delta_{l,k}$ are the corresponding cofactors. We now claim that each term in this sum has 
nontangential limit $0$ at $\ze$. Consider, for example, the last term $g^{(n)}\Delta_n$. First of all, 
since $g$ is nontangentially small of order $>\ga$ at $\ze$, Lemma \ref{lem:little} shows that $g^{(n)}$ is 
nontangentially small of order $>\ga-n$ at $\ze$; that is, 
\begin{equation}\label{eqn:pozor}
g^{(n)}(z)=o\left((1-|z|)^{\ga-n}\right)\quad\text{as}\quad|z|\to1,\,\,z\in\Ga_{\widetilde M}(\ze),
\end{equation}
for some $\widetilde M>1$. (We have applied the lemma with $G_0=\Ga_M(\ze)$ and $G=\Ga_{\widetilde M}(\ze)$, where 
$1<\widetilde M<M$. The hypothesis \eqref{eqn:ggg} is then fulfilled with a suitable $\de=\de(M,\widetilde M)$.) 
As to the cofactor 
$$\Delta_n=(-1)^{k+n}\det\left\{f_j^{(s)}:\,\,0\le j\le n\,\,(j\ne k),\,\,0\le s\le n-1\right\},$$ 
it is the sum of $n!$ products of the form 
\begin{equation}\label{eqn:prod}
\pm f_{j_1}f'_{j_2}\dots f_{j_n}^{(n-1)},
\end{equation}
where the multiindex $(j_1,\dots,j_n)$ runs through the permutations of 
$$(0,\dots,k-1,k+1,\dots,n).$$ 
For all $j$ and $s$, we have $f_j\in A^{-\be}$ and hence $f_j^{(s)}\in A^{-\be-s}$ (apply Lemma \ref{lem:decay} 
with $G=G_0=\D$, $\de=\f12$ and $\al=-\be$), so that 
$$f_j^{(s)}(z)=O\left((1-|z|)^{-\be-s}\right),\qquad z\in\D.$$ 
It follows that the products \eqref{eqn:prod} are all $O\left((1-|z|)^{-\kappa}\right)$, where 
$$\kappa=n\be+1+2+\dots+(n-1)=n\be+\f{n(n+1)}2-n=\ga-n.$$ 
A similar estimate therefore holds for $\Delta_n$; thus 
$$\Delta_n(z)=O\left((1-|z|)^{n-\ga}\right),\qquad z\in\D.$$ 
Combining this with \eqref{eqn:pozor}, we obtain 
$$g^{(n)}(z)\Delta_n(z)\to0\quad\text{as}\quad|z|\to1,\,\,z\in\Ga_{\widetilde M}(\ze).$$ 
\par The other terms on the right-hand side of \eqref{eqn:wrok} are treated similarly, and we conclude that 
$W_k$ has nontangential limit $0$ at $\ze$. The same is then true for $W=W_k/\la_k$, so we see that 
$E_\ga(f_0,\dots,f_n)$ is a subset of 
$$\left\{\ze\in\T:\,W(z)\to0\text{ as }z\to\ze\text{ nontangentially}\right\}.$$ 
This last set has Lebesgue measure $0$ by virtue of the Lusin--Privalov uniqueness theorem (see \cite{CL}), and the 
required result follows. 
\quad\qed

\medskip\noindent{\it Remark.} In the above proof, Lemma \ref{lem:little} was applied to the case where $G$ 
and $G_0$ are two Stolz angles with the same vertex. Other choices of $G$ and $G_0$, with \eqref{eqn:ggg} fulfilled, 
may lead to further variations on the theme of Theorem \ref{thm:kornontang}. One such choice can be described as follows: 
given a function $h\in H^\infty$ with $\|h\|_\infty=1$, fix a number $\eps\in(0,1)$ and take the level set 
$$\Om(h,\eps):=\{z\in\D:\,|h(z)|<\eps\}$$ 
as $G_0$, then put $G=\Om(h,\eps/2)$. The situation becomes nontrivial if the closure of $G$ hits $\T$. We shall return 
to this in the next section; in particular, see Lemma \ref{lem:levset} below. 

\section{Smooth analytic functions and inner factors} 

By a \lq\lq smooth analytic function" we mean a function in $\mathcal H(\D)$ that is smooth up to $\T$. 
Specifically, we are concerned here with the analytic Lipschitz--Zygmund spaces $\Al$ for $\al>0$. Recall that 
a function $f\in\mathcal H(\D)$ is said to be in $\Al$ if it obeys condition \eqref{eqn:lipkor}, where $m$ 
is the smallest integer in $(\al,\infty)$. In fact, taking $m$ to be any other integer in that interval, 
one arrives at an equivalent definition. It should be mentioned that, in the case $\al\in(0,\infty)\setminus\N$, 
an analytic function $f$ will be in $\Al$ if and only if there exists a constant $C=C_f$ such that 
\begin{equation}\label{eqn:lipond}
|f^{(k)}(z)-f^{(k)}(w)|\le C|z-w|^{\al-k},\qquad z,w\in\D,
\end{equation}
where $k=[\al]$ is the integral part of $\al$; this is a classical result of Hardy and Littlewood. 
The space $A^1$, known as the analytic {\it Zygmund class}, can be described by the appropriate 
second order difference condition on $f$. The higher order Zygmund classes $A^k$ with $k=2,3,\dots$ are 
related to it by the formula $A^k=\{f:f^{(k-1)}\in A^1\}$. 
\par We further recall that a closed subset $E$ of $\D\cup\T$ will be the zero set of some non-null 
function in $\Al$, with any fixed $\al>0$, or in $A^\infty:=\bigcap_{0<\al<\infty}\Al$ if and only if 
it has the two properties below: first, 
\begin{equation}\label{eqn:blacond}
\sum_{z\in E\cap\D}(1-|z|)<\infty
\end{equation}
(i.e., $E\cap\D$ satisfies the {\it Blaschke condition}), and second, 
\begin{equation}\label{eqn:carlcond}
\int_\T\log\text{\rm dist}(\ze,E)\,|d\ze|>-\infty
\end{equation}
(which is known as the {\it Carleson condition}). This characterization is due to Carleson \cite{C} in the case 
where $E\subset\T$, with $\al$ finite, and to Taylor and Williams \cite{TW} in the general case. 
\par In what follows, a subset $E$ (not necessarily closed) of $\D\cup\T$ will be 
called a {\it (BC)-set} if it satisfies \eqref{eqn:blacond} and \eqref{eqn:carlcond}. 
In other words, $E$ is a (BC)-set if and only if its closure, $\text{\rm clos}\,E$, is a zero set for 
$\Al$ or $A^\infty$. The closed (BC)-sets that are contained in $\T$ are called {\it Carleson sets}; these are 
described by condition \eqref{eqn:carlcond} alone. Of course, such sets have Lebesgue measure $0$ on $\T$. 
\par Now suppose that $n\in\N$ and $f\in\Al$ with $\al>n$. Then $f^{(n)}$ is continuous up to $\T$, and there is 
an obvious way to extend the notion of an $n$-deep zero to a boundary point: just say that $f$ has an $n$-deep zero 
at a point $\ze\in\T$ if $f^{(l)}(\ze)=0$ for $l=0,\dots,n$. This done, we arrive at the following version of 
Theorem A for $\Al$ functions. 

\begin{thm}\label{thm:trivlip} Suppose $f_0,\dots,f_n$ are linearly independent functions in $\Al$, where $\al>n$. 
Then there is a (BC)-set $\mathcal E$ with the following property: whenever $\la_0,\dots,\la_n$ are complex numbers 
with $\sum_{j=0}^n|\la_j|>0$, the $n$-deep zeros of the function $\sum_{j=0}^n\la_jf_j$ in $\D\cup\T$ are all contained 
in $\mathcal E$. 
\end{thm}

The proof is essentially the same as that of Theorem A. The only new feature is that the exceptional set $\mathcal E$, 
defined again as the zero set of $W:=W(f_0,\dots,f_n)$, will now be a (BC)-set. This is due to the fact that, under 
the current assumptions, $W$ is a nontrivial function in $A^{\al-n}$. 
\par Our next result, Theorem \ref{thm:fullmult} below, is similar in nature but deals with a somewhat 
more sophisticated situation. This time, the boundary smallness condition imposed on $g=\sum_{j=0}^n\la_jf_j$ 
will be expressed by saying that $g$ multiplies (every power of) a certain inner function into $\Al$. Recall 
that a function $\th\in H^\infty$ is said to be {\it inner} if 
$$|\th(\ze)|=\lim_{r\to1^-}|\th(r\ze)|=1$$ 
at almost all points $\ze\in\T$. Further, given $f\in\Al$ and an inner function $\th$, we say that {\it $f$ is strongly 
multipliable by $\th$ in $\Al$} to mean that 
\begin{equation}\label{eqn:stromult}
f\th^k\in\Al\quad\text{\rm for all }\,k\in\N. 
\end{equation}
When $0<\al<1$, \eqref{eqn:stromult} is equivalent to the (formally) weaker condition that $f\th\in\Al$, but for 
larger $\al$'s this last condition becomes actually weaker in general; see \cite{Shi1, Shi2} and \cite{DSpb2010} for 
a discussion of this phenomenon. 
\par The following theorem, to be found in \cite{DSpb, DAmer}, provides a criterion for a function 
$f\in\Al$ to be strongly multipliable by an inner function $\th$ in $\Al$. See also 
\cite{DActa, DAdv, Dyn2} for alternative versions and approaches. The criterion will be stated 
in terms of a decrease condition to be satisfied by $f$ along the set 
$$\Omte:=\{z\in\D:|\th(z)|<\eps\},\qquad0<\eps<1.$$ 

\begin{thb} Let $0<\al<\infty$ and let $m$ be an integer with $m>\al$. Given $f\in\Al$ and 
an inner function $\th$, the following conditions are equivalent. 

\smallskip{\rm (i.B)} $f\th^m\in\Al$. 

\smallskip{\rm (ii.B)} $f$ is strongly multipliable by $\th$ in $\Al$. 

\smallskip{\rm (iii.B)} For some $\eps\in(0,1)$, one has 
$$f(z)=O\left((1-|z|)^\al\right),\qquad z\in\Omte.$$ 
\end{thb}

Yet another equivalent condition is obtained from (iii.B) upon replacing the word \lq\lq some" 
by \lq\lq each"; see \cite{DSpb} or \cite{DAmer}. The constant in the $O$-condition will, of course, 
depend on $\eps$. Finally, we remark that if any of the conditions (i.B)--(iii.B) holds for an 
$f\in\Al$, $f\not\equiv0$, with a nontrivial inner function $\th$ (where \lq\lq nontrivial" means 
distinct from a finite Blaschke product), then $f$ vanishes on 
\begin{equation}\label{eqn:sigth}
\si(\th):=\T\cap\text{\rm clos}\,\Om\left(\th,\f12\right),
\end{equation}
and the latter is therefore a Carleson set. This set $\si(\th)$ is called the {\it boundary spectrum} 
of $\th$; the value $\f12$ in \eqref{eqn:sigth} may safely be replaced by any other $\eps\in(0,1)$. 
\par The main result of this section is as follows. 

\begin{thm}\label{thm:fullmult} Let $f_0,\dots,f_n$ be linearly independent functions in $\Al$, 
where $\al>n$. If $g$ is a nontrivial linear combination of the $f_j$'s, and if $\th$ is an inner 
function such that $g$ is strongly multipliable by $\th$ in $\Al$, then $W:=W(f_0,\dots,f_n)$ is 
strongly multipliable by $\th$ in $A^{\al-n}$. 
\end{thm}

Thus, the inner functions $\th$ that arise in connection with {\it all} the nontrivial $g$'s as 
above (in the sense that $g$ is strongly multipliable by $\th$ in $\Al$) are actually related 
in a similar way, but with $\al-n$ in place of $\al$, to a {\it single} function $W$ in $A^{\al-n}$. 
In particular, {\it the boundary spectrum $\si(\th)$ of any such $\th$ is contained in a fixed Carleson 
set} (namely, in the boundary zero set of $W$). This last property can no longer be guaranteed if we modify 
the hypotheses of Theorem \ref{thm:fullmult} by taking $\al=n$, as the following example shows. 

\medskip\noindent{\it Example.} Let 
$$B_1(z):=\prod_{j=1}^\infty\f{1-2^{-j}-z}{1-(1-2^{-j})z},\qquad z\in\D,$$ 
so that $B_1$ is the Blaschke product with zeros $\{1-2^{-j}\}$. It is fairly easy to check that, 
for $0<\eps<1$, the level set $\Om(B_1,\eps)$ is contained in some Stolz angle $\Ga_M(1)$. Consequently, 
given $\ze\in\T$, the Blaschke product $B_\ze(z):=B_1(\bar\ze z)$ will have its level set $\Om(B_\ze,\eps)$ 
in $\Ga_M(\ze)$. The function $g_\ze(z):=(z-\ze)^n$ is a linear combination of $1,\dots,z^n$ and is 
$O\left((1-|z|)^n\right)$ on $\Ga_M(\ze)$, so we deduce from Theorem B that $g_\ze$ is strongly multipliable 
by $B_\ze$ in $A^n$. On the other hand, we have $\si(B_\ze)=\{\ze\}$ and hence $\bigcup_{\ze\in\T}\si(B_\ze)=\T$. 

\medskip To prove Theorem \ref{thm:fullmult}, the following elementary lemma will be needed. 

\begin{lem}\label{lem:levset} Let $h\in H^\infty$, $\|h\|_\infty\le1$, and let $0<\eps<1$. Then, for 
every $z\in\Om(h,\f\eps2)$, the disk $B(z,\f\eps4(1-|z|))$ is contained in $\Om(h,\eps)$. 
\end{lem}

\begin{proof} One readily checks that $B(z,\f\eps4(1-|z|))$ is contained in the non-Euclidean disk 
$$K\left(z,\f\eps4\right):=\left\{w\in\D:\,\rho(z,w)<\f\eps4\right\},$$
where $\rho(\cdot,\cdot)$ is the pseudohyperbolic distance on $\D$ given by 
$$\rho(z,w):=\left|\f{z-w}{1-\bar wz}\right|.$$ 
Now, given $z\in\Om(h,\f\eps2)$ and $w\in K\left(z,\f\eps4\right)$, we use the well-known inequality 
$$\rho(h(z),h(w))\le\rho(z,w)$$ 
(see \cite[Chapter I]{G}) to deduce that 
$$|h(z)-h(w)|\le2\rho(z,w)<\f\eps2$$ 
and hence 
$$|h(w)|\le|h(z)|+\f\eps2<\eps.$$ 
This shows that $K\left(z,\f\eps4\right)\subset\Om(h,\eps)$, and the required conclusion follows. 
\end{proof}

\medskip\noindent{\it Proof of Theorem \ref{thm:fullmult}.} Let $\th$ be inner, and suppose 
\begin{equation}\label{eqn:glico}
g=\sum_{j=0}^n\la_jf_j 
\end{equation}
is a non-null function that is strongly multipliable by $\th$ in $\Al$. As in the proof of 
Theorem \ref{thm:kornontang} above, we now fix an index $k\in\{0,\dots,n\}$ with $\la_k\ne0$ 
and consider the Wronskian 
\begin{equation}\label{eqn:wrog}
W_k:=W(f_0,\dots,f_{k-1},g,f_{k+1},\dots,f_n),
\end{equation}
so that $W_k=\la_kW$. We also recall that $W\not\equiv0$ and $W\in A^{\al-n}$; the latter is 
due to the fact that all the entries of the corresponding Wronskian matrix are in $A^{\al-n}$. 
\par Expanding, as before, the determinant \eqref{eqn:wrog} along the column 
$\left(g,g',\dots,g^{(n)}\right)^T$, we obtain 
\begin{equation}\label{eqn:wrokkk}
W_k=\sum_{l=0}^ng^{(l)}\Delta_l,
\end{equation}
where $\Delta_l$ are the appropriate cofactors. Since $g$ is strongly multipliable by $\th$ in $\Al$, 
it follows from Theorem B that 
$$g(z)=O\left((1-|z|)^\al\right),\qquad z\in\Omte,$$ 
for some $\eps\in(0,1)$. In view of Lemma \ref{lem:decay}, this yields 
$$g^{(l)}(z)=O\left((1-|z|)^{\al-l}\right),\qquad z\in\Om\left(\th,\f\eps2\right),$$ 
for each $l\in\N$. (We have applied Lemma \ref{lem:decay} with $G_0=\Omte$ and $G=\Om(\th,\eps/2)$. 
The hypothesis \eqref{eqn:ggg} is then fulfilled with $\de=\eps/4$, as Lemma \ref{lem:levset} shows.) 
In particular, 
\begin{equation}\label{eqn:frog}
g^{(l)}(z)=O\left((1-|z|)^{\al-n}\right),\qquad z\in\Om\left(\th,\f\eps2\right),\quad 0\le l\le n.
\end{equation}
The cofactors $\Delta_l$ are all in $A^{\al-n}$, and hence in $H^\infty$, so \eqref{eqn:wrokkk} 
and \eqref{eqn:frog} together imply that 
$$W_k(z)=O\left((1-|z|)^{\al-n}\right),\qquad z\in\Om\left(\th,\f\eps2\right).$$ 
A similar estimate holds then for $W=W_k/\la_k$, and another application of Theorem B convinces us 
that $W$ is strongly multipliable by $\th$ in $A^{\al-n}$, as desired. 
\quad\qed

\medskip We supplement Theorem \ref{thm:fullmult} with the next result, which involves 
the {\it star-invariant subspace} 
\begin{equation}\label{eqn:kth}
K_\th:=H^2\ominus\th H^2
\end{equation}
of the Hardy space $H^2$. Here, the term \lq\lq star-invariant" means invariant under the backward shift 
operator $f\mapsto(f-f(0))/z$. It is well known that the (closed and proper) star-invariant subspaces of $H^2$ 
are precisely those of the form \eqref{eqn:kth}, with $\th$ an inner function; see \cite{N}. Also, the Korenblum 
spaces $A^{-\be}$ from the previous section will now reappear, along with the Lipschitz--Zygmund $\Al$ spaces. 

\begin{thm}\label{thm:wrokth} Let $f_0,\dots,f_n$ be linearly independent functions in $\Al$, where $\al>n$. 
Further, let $F\in A^{-\be}$ with $\be=\al-n$, and let $\th$ be an inner function such that $FW\in K_\th$, 
where $W:=W(f_0,\dots,f_n)$. Assume finally that there exists a nontrivial linear combination of the $f_j$'s 
which is strongly multipliable by $\th$ in $\Al$. Then $FW\in H^\infty$. 
\end{thm}

The statement becomes especially transparent when $n=0$, in which case it reduces to the following. 

\begin{cor}\label{cor:spec} Given $0<\al<\infty$, suppose that $F\in A^{-\al}$, $g\in\Al$, and $\th$ is an 
inner function. If $g$ is strongly multipliable by $\th$ in $\Al$, and if $Fg\in K_\th$, then $Fg\in H^\infty$. 
\end{cor}

We now cite, as Lemma \ref{lem:mpcohn} below, a remarkable \lq\lq maximum principle" for $K_\th$ that was 
proved by Cohn in \cite{Co1}. 

\begin{lem}\label{lem:mpcohn} Let $\th$ be inner, and let $f\in K_\th$. If 
\begin{equation}\label{eqn:leveps}
\sup\{|f(z)|:\,z\in\Omte\}<\infty
\end{equation}
for some $\eps\in(0,1)$, then $f\in H^\infty$.
\end{lem}

\medskip\noindent{\it Proof of Theorem \ref{thm:wrokth}.} Proceeding as in the proof of Theorem \ref{thm:fullmult}, 
we verify that $W:=W(f_0,\dots,f_n)$ is strongly multipliable by $\th$ in $A^{\be}$, so that 
$$W(z)=O\left((1-|z|)^{\be}\right),\qquad z\in\Omte,$$ 
with a suitable $\eps\in(0,1)$. Since $F\in A^{-\be}$, we also have 
$$F(z)=O\left((1-|z|)^{-\be}\right),\qquad z\in\D.$$ 
Consequently, the function $f:=FW$ satisfies \eqref{eqn:leveps}, and 
Lemma \ref{lem:mpcohn} now tells us that $f\in H^\infty$. 
\quad\qed

\section{Hardy--Sobolev spaces: deep zeros and singular factors} 

\par Here, we shall be concerned with functions in the {\it Hardy--Sobolev spaces} 
$$H^p_k:=\{f\in\mathcal H(\D):\,f^{(k)}\in H^p\},$$ 
with $p>0$ and $k\in\N$, where $H^p=H^p(\D)$ are the usual (holomorphic) {\it Hardy spaces} on the 
disk. 
\par It is well known that $H^p_1$ is contained in the Lipschitz space $A^{1-1/p}$ for $p>1$, while 
\begin{equation}\label{eqn:diskalg}
H^1_1\subset H^\infty
\end{equation}
(moreover, the functions from $H^1_1$ are continuous up to $\T$) and 
\begin{equation}\label{eqn:hardlit}
H^p_1\subset H^{p/(1-p)}\quad\text{\rm for}\quad0<p<1.
\end{equation}
These results, which are chiefly due to Hardy and Littlewood, can be found in \cite{Du}. Iterating them, one arrives 
at the appropriate embedding theorems for $H^p_k$ with $k\ge2$. In particular, we always have $H^p_k\subset H^p$. 

\par Now let us recall that any nontrivial function $f\in H^p$ can be factored canonically 
as $f=I\mathcal O$, where $I$ is inner and $\mathcal O$ is outer. (A function $\mathcal O\in\mathcal H(\D)$ 
is called {\it outer} if $\log|\mathcal O(z)|$ agrees, for $z\in\D$, with the harmonic 
extension of an integrable function on $\T$.) The inner factor $I$ can be further decomposed 
as $I=BS$, where $B$ is a {\it Blaschke product} and $S$ is a {\it singular inner function}; 
see \cite[Chapter II]{G}. More explicitly, the factors $B$ and $S$ in this last formula are of the form 
$$B(z)=B_{\{z_j\}}(z):=\prod_j\f{\bar z_j}{|z_j|}\f{z_j-z}{1-\bar z_jz},$$ 
where $\{z_j\}\sb\D$ is a sequence (possibly finite or empty) with $\sum_j(1-|z_j|)<\infty$, and 
$$S(z)=S_\mu(z):=\exp\left\{-\int_\T\f{\ze+z}{\ze-z}d\mu(\ze)\right\},$$
where $\mu$ is a (nonnegative) singular measure on $\T$. 

\par We further remark that if the functions $f_0,\dots,f_n$ from Theorem A are taken 
to be in $H^p_n$, then their Wronskian $W$ is in $H^q$ with a suitable $q$; in case $p\ge1$, 
this is true with $q=p$. (To verify these claims, use \eqref{eqn:diskalg} and \eqref{eqn:hardlit}.) 
The zero set $\mathcal E=W^{-1}(0)$ therefore satisfies the Blaschke 
condition $\sum_{z\in\mathcal E}(1-|z|)<\infty$, and we may rephrase Theorem A as saying that there 
exists a Blaschke product (the one built from $\mathcal E$) with a certain divisibility property. 
Our current purpose is to extend this to singular inner factors; in a sense, such factors can be 
thought of as responsible for the function's \lq\lq boundary zeros of infinite multiplicity". 

\par A unified statement, involving both Blaschke products and singular factors, is given in 
Theorem \ref{thm:genmainsing} below. In particular, it turns out that if the linearly independent functions 
$f_0,\dots,f_n$ are in $H^1_n$, then (much in the spirit of Theorem A) there is a {\it single} singular inner 
function $S$ divisible by the singular factor of each nontrivial linear combination $\sum_{j=0}^n\la_jf_j$. 
This means that the totality of singular factors resulting from such linear combinations is rather poor. 
On the other hand, we shall see that the hypothesis $f_j\in H^1_n$, or at least some smoothness assumption on 
the $f_j$'s, is indispensable; in fact, it is not enough to assume that the functions are merely in $H^\infty$. 

\par Let $I=BS$ be an inner function, where $B$ is a Blaschke product and $S$ is singular, and let $n\in\N$. 
We write $B_{>n}$ for the Blaschke product obtained from $B$ by removing the zeros of multiplicity $\le n$ 
(the remaining zeros, if any, are kept with their multiplicities unchanged); then we put $I_{>n}=B_{>n}S$. 

\begin{thm}\label{thm:genmainsing} Let $f_0,\dots,f_n$ be linearly independent functions in $H^1_n$. Then 
there is an inner function $J$ with the following property: whenever $\la_0,\dots,\la_n$ are complex 
numbers with $\sum_{j=0}^n|\la_j|>0$ and $I$ is the inner factor of $\sum_{j=0}^n\la_jf_j$, the inner 
function $I_{>n}$ divides $J$. 
\end{thm}

\begin{proof} Because the $f_j$'s are linearly independent, the Wronskian
$$W:=W(f_0,\dots,f_n)$$ 
is non-null. In addition, $W\in H^1$. To verify this last fact, expand the determinant \eqref{eqn:wro} along its 
last row and observe that $f_j^{(n)}\in H^1$ for each $j$, while the derivatives $f_j^{(\nu)}$ with $0\le\nu\le n-1$ are 
all in $H^\infty$. Consider the inner factor $B_WS_W$ of $W$; here $B_W$ is a Blaschke product and $S_W$ is singular. 
Further, let $z_l$ ($l=1,2,\dots$) be the distinct zeros of $B_W$, of respective multiplicities $m_l$. We now 
form a new Blaschke product $\widetilde B$ with the same zero set $\{z_l\}$, this time assigning multiplicity 
$m_l+n$ to $z_l$. Our plan is to check that the inner function $J:=\widetilde BS_W$ has the required property. 

\par Let 
\begin{equation}\label{eqn:ggggg} 
g=\sum_{j=0}^n\la_jf_j
\end{equation}
be a nontrivial linear combination of the $f_j$'s, and let $I=I_g=B_gS_g$ be the inner factor of $g$. (Again, it is 
understood that $B_g$ is a Blaschke product and $S_g$ is singular.) Next, we fix an index $k\in\{0,\dots,n\}$ for 
which the coefficient $\la_k$ in \eqref{eqn:ggggg} is nonzero, and we put 
$$W_k:=W(f_0,\dots,f_{k-1},g,f_{k+1},\dots,f_n).$$ 
Thus, $W_k$ is the determinant obtained from \eqref{eqn:wro} by replacing its $k$th column with 
\begin{equation}\label{eqn:gcol}
\left(g,g',\dots,g^{(n)}\right)^{T},
\end{equation}
and it follows from \eqref{eqn:ggggg} that $W_k=\la_kW$. In particular, the inner factors of $W$ and $W_k$ are 
identical. 

\par Now assume that $g$ has a zero of multiplicity $\mu$, $\mu>n$, at a point $\ze\in\D$ (so that $\ze$ is 
a zero of $I_{>n}$). The derivative $g^{(\nu)}$, with $\nu=1,\dots,n$, will then vanish to order $\mu-\nu$ at $\ze$. 
Therefore, expanding the determinant $W_k$ along its $k$th column \eqref{eqn:gcol}, we see that $W_k$ (and hence $W$, 
as well as $B_W$) has a zero of multiplicity at least $\mu-n$ at $\ze$. Consequently, $\ze$ coincides with 
$z_l$ for some $l$, and $\mu-n\le m_l$. Thus $\mu$ does not exceed $m_l+n$ (which is the multiplicity of 
$z_l$ as a zero of $\widetilde B$), and we deduce that $\widetilde B$ is divisible by $b_\ze^\mu$, where 
$b_\ze(z)=\f{\ze-z}{1-\bar\ze z}$. This shows that the Blaschke factor of $I_{>n}$ divides $\widetilde B$, the 
Blaschke factor of $J$. 

\par Finally, we need to deal with the singular parts of $I_{>n}$ and $J$. Specifically, we must check 
that $S_g$ divides $S_W$. To this end, we notice that the inner factors of $g',g'',\dots,g^{(n)}$ are all 
divisible by $S_g$. (In fact, it is known that for every $h\in H^1_1$, the singular factor of $h$ divides 
that of $h'$; see \cite{Cau} or \cite{VS1}.) Once again, we expand the determinant $W_k$ along its $k$th 
column \eqref{eqn:gcol}, while noting that the corresponding cofactors are in $H^1$, to conclude that 
the singular inner factor of $W_k$ (i.e., $S_W$) is indeed divisible by $S_g$. The proof is now complete. 
\end{proof}

\par When specialized to singular factors, Theorem \ref{thm:genmainsing} takes a simpler form. 

\begin{thm}\label{thm:mainsing} Let $f_0,\dots,f_n$ be linearly independent functions in $H^1_n$. 
Then there is a singular inner function $S$ with the following property: whenever $\la_0,\dots,\la_n$ 
are complex numbers with $\sum_{j=0}^n|\la_j|>0$, the singular factor of $\sum_{j=0}^n\la_jf_j$ divides $S$. 
\end{thm}

\par From this, yet another fact will be deduced. But first we cite, as Lemma \ref{lem:powerwro} below, 
a somewhat restricted version of a result from \cite{Z}. With a further application in mind, we state it 
for a generic domain $\Om\subset\C$ rather than for the disk. 

\begin{lem}\label{lem:powerwro} If $f\in\mathcal H(\Om)$, then 
$$W(1,f,f^2,\dots,f^n)=c_n\left[f'\right]^{n(n+1)/2},$$
where $c_n=\prod_{k=1}^nk!$. 
\end{lem}

\par To derive the next corollary, it suffices to apply Theorem \ref{thm:mainsing} to the case where the 
$f_j$'s are powers of a single function. In doing so, one should take the function $S$ in that theorem to 
be the singular factor of the Wronskian (in accordance with the preceding proof) and combine this with 
Lemma \ref{lem:powerwro}. 

\begin{cor}\label{cor:powersing} Let $f$ be a nonconstant function in $H^1_n$, and let $\mathbf S$ be the 
singular inner factor of $f'$. Then $\mathbf S^{n(n+1)/2}$ is divisible by the singular inner factor of 
every linear combination $\sum_{k=0}^n\la_kf^k$ with $\sum_{k=0}^n|\la_k|>0$. In particular, if $f'$ has no 
singular factor, then the same is true for each of the linear combinations in question. 
\end{cor}

\par Finally, we show that Theorem \ref{thm:mainsing} (and hence also Theorem \ref{thm:genmainsing}) becomes 
false, already for $n=1$, if we replace $H^1_n$ by $H^\infty$. 

\begin{prop}\label{prop:hinftyfalse} There are functions $f_0,f_1\in H^\infty$ with the following 
property: for each singular inner function $S$ there is a nontrivial linear combination $\la_0f_0+\la_1f_1$ 
whose singular factor does not divide $S$. 
\end{prop}

\begin{proof} We borrow an idea from \cite{CS}. Let $\th$ be a nonconstant inner function that omits 
an uncountable set of values $\mathcal A\subset\D$. (The existence of such a function with values 
in $\D\setminus\mathcal A$, for any prescribed closed set $\mathcal A$ of zero logarithmic capacity, 
is established in \cite[Chapter 2]{CL}.) For each $\al\in\mathcal A$, one has 
\begin{equation}\label{eqn:tetal}
\th-\al=S_\al\cdot(1-\bar\al\th),
\end{equation}
with 
$$S_\al:=\f{\th-\al}{1-\bar\al\th}.$$ 
Here, $S_\al$ is a singular inner function (because $\al$ is not in the range of $\th$), while the other factor 
in \eqref{eqn:tetal} is outer. 
\par Write $\mu_\al$ for the singular measure associated with $S_\al$. For $\mu_\al$-almost every $\ze\in\T$, 
we have $S_\al(z)\to0$, and hence $\th(z)\to\al$, as $z\to\ze$ nontangentially; see \cite[Chapter II]{G}. It 
follows that the supports of $\mu_\al$'s, with $\al\in\mathcal A$, are pairwise disjoint. The set $\mathcal A$ being 
uncountable and the measures $\mu_\al$ nonzero, we readily deduce that no finite Borel measure $\mu$ on $\T$ can satisfy 
$\mu\ge\mu_\al$ for all $\al\in\mathcal A$. Consequently, no singular inner function is divisible by every $S_\al$. 
\par This said, we put $f_0:=\th$ and $f_1:=1$. Since $S_\al$ is the singular factor of $\th-\al$, which is 
a linear combination of $f_0$ and $f_1$, we are done. 
\end{proof}

\par It would be interesting to know if the space $H^1_n$ in Theorems \ref{thm:genmainsing} and \ref{thm:mainsing} 
can be replaced by a larger smoothness class (say, by $H^p_n$ with a $p<1$), and moreover, to find the optimal 
smoothness conditions on the functions $f_j$ that guarantee the validity of those results. 

\section{Zero sets of Wronskians} 

In this section, we discuss the exceptional sets $\mathcal E$ that may arise in Theorem A (see Introduction) 
when the functions $f_0,\dots,f_n$ are assumed to lie in a specific space $X$ of holomorphic functions 
on a domain $\Om\subset\C$. For the time being, we prefer to deal with a general domain, not necessarily 
with the disk. Given a space $X\subset\mathcal H(\Om)$, we now introduce the appropriate concepts and notations. 

\medskip\noindent\textbf{Definition.} (a) Let $f_0,\dots,f_n$ be linearly independent functions in $\mathcal H(\Om)$. 
We write $\mathcal Z(f_0,\dots,f_n)$ for the set of points $z\in\Om$ with the following property: there exist complex 
numbers $\la_0=\la_0(z),\,\dots,\,\la_n=\la_n(z)$ with $\sum_{j=0}^n|\la_j|>0$ such that the function 
$\sum_{j=0}^n\la_jf_j$ has an $n$-deep zero at $z$. 
\par (b) We denote by $\mathcal Z_n(X)$ the collection of those sets $E\subset\Om$ which can be written as 
$E=\mathcal Z(f_0,\dots,f_n)$ for some linearly independent functions $f_0,\dots,f_n$ in $X$. 

\medskip The proof of Theorem A from the Introduction shows that $\mathcal Z(f_0,\dots,f_n)$ is no other than 
the zero set of $W(f_0,\dots,f_n)$. We further remark that, when $n=0$, $\mathcal Z(f_0)$ is just the zero set 
of $f_0$, while $\mathcal Z_0(X)$ is the class of zero sets for $X$. 
\par In what follows, we write $\mathcal Z(X)$ for $\mathcal Z_0(X)$; a (discrete) set $E$ in $\Om$ will thus 
belong to $\mathcal Z(X)$ if and only if it coincides with $\mathcal Z(g):=g^{-1}(0)$ for some 
non-null function $g\in X$. Also, we put 
\begin{equation}\label{eqn:xprime}
X':=\{f':\,f\in X\}.
\end{equation}

\begin{thm}\label{thm:zezeze} Let $X\subset\mathcal H(\Om)$ and $n\in\N$. Assume, in addition, that $X$ is an algebra (with 
respect to the usual pointwise multiplication of functions) that contains the constant function $1$ and satisfies $X'=X$. 
Then $\mathcal Z_n(X)=\mathcal Z(X)$. 
\end{thm}

\begin{proof} Let $E\in\mathcal Z_n(X)$, so that $E$ is the zero set of $W:=W(f_0,\dots,f_n)$ for some linearly independent 
functions $f_0,\dots,f_n\in X$. Our assumptions on $X$ imply that $W\in X$ (because all the entries of the Wronskian matrix 
are in $X$), and so $E\in\mathcal Z(X)$. This proves the inclusion $\mathcal Z_n(X)\subset\mathcal Z(X)$. 
\par Conversely, suppose $E\in\mathcal Z(X)$, so that $E=g^{-1}(0)$ for some non-null function $g\in X$. We then write 
$g=f'$ for a suitable $f\in X$ and invoke Lemma \ref{lem:powerwro} to get 
$$W(1,f,\dots,f^n)=c_ng^{n(n+1)/2}.$$ 
This last Wronskian (which is built from the functions $f^k$ lying in $X$) vanishes precisely on $E$, so 
$E\in\mathcal Z_n(X)$; this shows that $\mathcal Z(X)\subset\mathcal Z_n(X)$. 
\end{proof}

As examples of algebras $X$ satisfying the hypotheses of Theorem \ref{thm:zezeze}, with $\Om=\D$, we mention 
$A^\infty:=\bigcap_{\al>0}\Al$ and $A^{-\infty}:=\bigcup_{\be>0}A^{-\be}$. In the former case, the family $\mathcal Z(X)$ 
is formed by the closed (BC)-sets (see Section 3 above), while in the latter case it is characterized by Korenblum's density 
condition (see \cite{K}). To give yet another example, this time with $\Om=\C$, fix a number $\rho\in(0,\infty)$ and 
take $X$ to be the space of entire functions of order at most $\rho$ and of finite type. For a description of $\mathcal Z(X)$ 
in this last example, we refer to \cite[Chapter I]{Le}. 
\par For the rest of the paper, we go back to the case $\Om=\D$. 

\begin{thm}\label{thm:hpm} Let $m$ and $n$ be nonnegative integers, and let $0<p<\infty$. 
\par{\rm (i)} In order that every set $E$ in $\mathcal Z_n(H^p_m)$ satisfy the Blaschke condition 
$$\sum_{z\in E}(1-|z|)<\infty,$$ 
it is necessary and sufficient that $m\ge n$. 
\par{\rm (ii)} In order that every set in $\mathcal Z_n(H^p_m)$ be a (BC)-set, it is necessary and sufficient 
that $m\ge n+p^{-1}$. 
\end{thm}

A few preliminary results will be needed. 

\begin{lem}\label{lem:embed} Given $p>0$ and $k\in\N$ with $kp\ge1$, one has $H^p_k\subset H^1_1$. 
\end{lem}

Indeed, the case $p\ge1$ is trivial in view of \eqref{eqn:diskalg}, while for $0<p<1$ the required fact 
can be established by repeated application of \eqref{eqn:hardlit}. 

\begin{lem}\label{lem:honeone} The space $H^1_1$ is an algebra, and every zero set for $H^1_1$ is a (BC)-set. 
\end{lem}

Here, the first statement is an easy consequence of \eqref{eqn:diskalg}; the second, which is much deeper, 
was proved in \cite{VS2}. See also \cite{CN} in connection with boundary zero sets. 

\begin{lem}\label{lem:nonbla} For each $l\in\N$ and $p>0$, the space 
$$\left(H^p\right)^{(l)}:=\left\{f^{(l)}:\,f\in H^p\right\}$$
contains a function whose zero set fails to satisfy the Blaschke condition. 
\end{lem}

To verify this, recall that even the Bloch space $\mathcal B=(A^1)'$ is known to contain functions with 
non-Blaschke zero sets (see \cite{ACP}). Since $A^1\subset H^\infty$, a similar conclusion holds for 
$(H^\infty)'$ and hence, {\it a fortiori}, for $\left(H^p\right)^{(l)}$. 

\par Finally, the next result is a restricted version of \cite[Theorem 1]{Li}. 

\begin{lem}\label{lem:linden} Let $k\in\N$ and $0<\al<1/(k+1)$. If $\{a_n\}\subset\D$ is a sequence such that 
\begin{equation}\label{eqn:blaal}
\sum_n(1-|a_n|)^\al<\infty
\end{equation}
and $B$ is the Blaschke product with zeros $\{a_n\}$, then $B\in H^p_k$ whenever $0<p\le(1-\al)/k$. 
\end{lem}

From this we deduce the following fact. 

\begin{cor}\label{cor:blapk} Let $k\in\N$ and $0<p<1/k$. If $B$ is a Blaschke product whose zero sequence $\{a_n\}$ 
satisfies $|a_n|=1-2^{-n}$ $(n=1,2,\dots)$, then $B\in H^p_k$. 
\end{cor}

Indeed, it suffices to apply Lemma \ref{lem:linden} with a suitably small $\al$, for instance, with 
$$\al=\min\left(1-pk,\,\f1{2(k+1)}\right).$$

\medskip\noindent{\it Proof of Theorem \ref{thm:hpm}.} We begin by proving the sufficiency in (i) and (ii). 
Consider the Wronskian $W:=W(f_0,\dots,f_n)$ of some linearly independent functions $f_0,\dots,f_n\in H^p_m$. 
\par Now, if $m\ge n$, then $f_0^{(n)},\dots,f_n^{(n)}$ are in $H^p_k$, where $k=m-n$, and hence in $H^p$. It follows 
that the lower order derivatives $f_j^{(l)}$, with $0\le j\le n$ and $0\le l\le n-1$, are also (at least) in $H^p$. 
This in turn implies that $W$ lies in a certain $H^r$ space (recall that $\bigcup_{q>0}H^q$ is an algebra or simply 
put $r=p/(n+1)$ and use H\"older's inequality), so the zero set $\mathcal Z(W)$ satisfies the Blaschke condition. 
\par Similarly, if $m\ge n+p^{-1}$, then the $n$th derivatives $f_j^{(n)}$ of all the $f_j$'s are in $H^p_k$ with 
$k=m-n\ge p^{-1}$, and Lemma \ref{lem:embed} ensures that $f_j^{(n)}\in H^1_1$. From this we deduce that 
the lower order derivatives $f_j^{(l)}$ are also (at least) in $H^1_1$, which eventually yields $W\in H^1_1$ in 
view of Lemma \ref{lem:honeone}. The same lemma tells us, then, that $\mathcal Z(W)$ is a (BC)-set. 
\par Now let us turn to the necessity in (i) and (ii). Suppose that $m<n$ and put $\ell=n-m$. Further, invoke 
Lemma \ref{lem:nonbla} to find a function $g\in(H^p)^{(\ell)}$ whose zero set $\mathcal Z(g)$ is non-Blaschke, 
in the sense that 
$$\sum_{z\in\mathcal Z(g)}(1-|z|)=\infty,$$ 
and let $f\in\mathcal H(\D)$ be such that $f^{(n)}=g$. Since $g$ is the $\ell$th derivative of $f^{(m)}$, 
it follows that $f^{(m)}\in H^p$, or equivalently, $f\in H^p_m$. The elementary formula 
\begin{equation}\label{eqn:pussy}
W\left(1,\f z{1!},\,\dots,\,\f{z^{n-1}}{(n-1)!},f\right)=f^{(n)}
\end{equation}
shows that this last Wronskian vanishes precisely on $\mathcal Z(g)$, whence we conclude that $\mathcal Z_n(H^p_m)$ 
contains non-Blaschke sets. 
\par Finally, assume that $n\le m<n+p^{-1}$ and put $k=m-n$. Let $\{a_j\}$ be a sequence in $\D$ 
satisfying $|a_j|=1-2^{-j}$ $(j=1,2,\dots)$ and having the whole circle $\T$ as its limit set. Further, 
write $B$ for the Blaschke product with zeros $\{a_j\}$, and let $f\in\mathcal H(\D)$ be such that $f^{(n)}=B$. 
We have then 
$$f^{(m)}=f^{(n+k)}=B^{(k)}\in H^p$$ 
(by virtue of Corollary \ref{cor:blapk}), whence $f\in H^p_m$. Now, if $W$ stands for 
the Wronskian on the left-hand side of \eqref{eqn:pussy}, with our current $f$ plugged in, 
then \eqref{eqn:pussy} reduces to saying that $W=B$. Consequently, the zeros of $W$ in $\D$ are 
precisely the $a_j$'s, and these obviously fail to form a (BC)-set, since $\text{\rm clos}\,\{a_j\}\supset\T$. 
The proof is complete. 
\quad\qed

\par Before stating our last theorem, we have to introduce a bit of notation. Namely, given a space 
$X\subset\mathcal H(\D)$ and an integer $m\ge0$, we write $X_m$ for the set of those functions $f\in\mathcal H(\D)$ 
which satisfy $f^{(k)}\in X$ with $k=0,\dots,m$. Of course, if $X$ is the Hardy space $H^p$, then $X_m$ becomes 
$H^p_m$. The role of $X$ will alternatively be played by $\bmoa:=H^1\cap\bmo$ (where $\bmo$ is the space of 
functions of {\it bounded mean oscillation} on $\T$), as well as by the {\it Nevanlinna class} $\mathcal N$ and 
the {\it Smirnov class} $\mathcal N^+$. 
\par In connection with $\bmo$, the reader is referred to \cite[Chapter VI]{G}. The class $\mathcal N$ 
(resp., $\mathcal N^+$) is formed by the quotients $u/v$ with $u,v\in H^\infty$, where $v$ is 
zero-free (resp., outer) in $\D$; see \cite[Chapter II]{G} for equivalent definitions and characterizations. 
To keep on the safe side, we also recall the notation \eqref{eqn:xprime}, since this will be used again for some 
of our current $X$'s. 

\begin{thm}\label{thm:zazaza} Let $X$ be any of the following spaces: $\bmoa$, $H^p$ (with $0<p<\infty$), 
$\mathcal N^+$ or $\mathcal N$. Then, for every integer $m\ge0$, one has 
$$\mathcal Z_{m+1}(X_m)=\mathcal Z(\bmoa').$$ 
\end{thm}

While no explicit characterization of $\mathcal Z(\bmoa')$ seems to be available, it can be shown that 
$$\mathcal Z\left(A^{-1+\eps}\right)\subset\mathcal Z(\bmoa')\subset\mathcal Z\left(A^{-1}\right)$$
for an arbitrarily small $\eps>0$. On the other hand, the zero sets for an $A^{-\be}$ space with $\be>0$ 
are \lq\lq almost described" -- even though not completely describable -- by the appropriate Korenblum-type 
density condition; see Seip's refinements in \cite{Se1, Se2} to Korenblum's original work from \cite{K}. 
\par The proof of Theorem \ref{thm:zazaza} will make use of the following result, which we prove first. 

\begin{lem}\label{lem:nevbmo} The set 
$$\mathcal N\cdot\bmoa':=\{fg':\,f\in\mathcal N,\,g\in\bmoa\}$$ 
is a vector space that contains $\mathcal N'$.  
\end{lem}

\begin{proof} We begin by recalling the (well-known) fact that the space $\bmoa'$ is invariant under 
multiplication by $H^\infty$ functions. Indeed, if $g\in\bmoa$ and $h\in H^\infty$, then the measure 
$$|g'(z)|^2(1-|z|)\,dx\,dy$$
where $z=x+iy$, is a Carleson measure (see \cite[Chapter VI]{G}), and so is 
$$|g'(z)|^2|h(z)|^2(1-|z|)\,dx\,dy,$$
whence $g'h\in\bmoa'$. Thus, 
\begin{equation}\label{eqn:idbmoa}
H^\infty\cdot\bmoa'=\bmoa'.
\end{equation}
\par Now, to prove that $\mathcal N\cdot\bmoa'$ is a vector space, we need to verify the linearity property 
$$f_1,f_2\in\mathcal N\cdot\bmoa'\implies f_1+f_2\in\mathcal N\cdot\bmoa'.$$ 
To this end, we write 
$$f_j=\f{u_j}{v_j}\cdot w'_j\qquad(j=1,2),$$
where $u_j,v_j\in H^\infty$ and $w_j\in\bmoa$, and where $v_j$ is zero-free. Note that 
$$f_1+f_2=\f1{v_1v_2}\cdot\left(u_1v_2w'_1+u_2v_1w'_2\right).$$ 
Each of the two terms in brackets, and hence their sum, is then in $\bmoa'$ by virtue of \eqref{eqn:idbmoa}, 
while the factor $1/(v_1v_2)$ is in $\mathcal N$. 
\par Finally, to check that $\mathcal N'\subset\mathcal N\cdot\bmoa'$, take any $f\in\mathcal N$ and write 
$f=u/v$ with suitable $u,v\in H^\infty$, the function $v$ being zero-free. The formula 
\begin{equation}\label{eqn:pizdenka}
f'=\f1{v^2}\cdot(u'v-uv')
\end{equation}
now provides the sought-after factorization for $f'$, because $v^{-2}\in\mathcal N$ and $u'v-uv'\in\bmoa'$, 
the latter being a consequence of \eqref{eqn:idbmoa}. 
\end{proof} 

\par We mention in passing that Lemma \ref{lem:nevbmo} admits an extension to higher order derivatives. 
In addition, similar results are available for the Smirnov class $\mathcal N^+$. This, and more, can be found 
in \cite{DAASF}. See also \cite{Co2, CV} for the corresponding factorization theorems in the $H^p$ setting. 

\medskip\noindent{\it Proof of Theorem \ref{thm:zazaza}.} Since 
$$\bmoa\subset H^p\subset\mathcal N^+\subset\mathcal N,$$ 
we clearly have 
$$\mathcal Z_n(\bmoa_m)\subset\mathcal Z_n(H^p_m)\subset
\mathcal Z_n(\mathcal N^+_m)\subset\mathcal Z_n(\mathcal N_m)$$
for all $n$, and in particular for $n=m+1$. Consequently, it suffices to show that 
\begin{equation}\label{eqn:pizda1}
\mathcal Z(\bmoa')\subset\mathcal Z_{m+1}(\bmoa_m)
\end{equation}
and 
\begin{equation}\label{eqn:pizda2}
\mathcal Z_{m+1}(\mathcal N_m)\subset\mathcal Z(\bmoa').
\end{equation}
\par To check \eqref{eqn:pizda1}, assume that $E=\mathcal Z(g')$ for some $g\in\bmoa$, and 
let $f\in\mathcal H(\D)$ be such that $f^{(m)}=g$. Applying formula \eqref{eqn:pussy} with $n=m+1$ yields 
$$W\left(1,\f z{1!},\,\dots,\,\f{z^m}{m!},f\right)=g'.$$ 
The zero set of this last Wronskian is therefore $E$, and since the functions $z^k$ and $f$ are in $\bmoa_m$, 
it follows that $E\in\mathcal Z_{m+1}(\bmoa_m)$; this proves \eqref{eqn:pizda1}. 
\par To verify \eqref{eqn:pizda2}, consider the Wronskian determinant 
$$\mathcal W:=W(f_0,\dots,f_{m+1})$$ 
built from some (any) linearly independent functions $f_0,\dots,f_{m+1}$ in $\mathcal N_m$. Expanding the 
determinant along its last row, we get 
\begin{equation}\label{eqn:lobok}
\mathcal W=\sum_{j=0}^{m+1}f_j^{(m+1)}\Delta_j,
\end{equation}
where $\Delta_j$ are the appropriate cofactors. Because the derivatives $f_j^{(k)}$ with $0\le k\le m$ are all 
in $\mathcal N$, we see that the $\Delta_j$'s are also in $\mathcal N$, whereas the functions $f_j^{(m+1)}$ 
are in $\mathcal N'$. By Lemma \ref{lem:nevbmo}, for each $j\in\{0,\dots,m+1\}$ there are functions 
$\varphi_j\in\mathcal N$ and $\psi_j\in\bmoa$ such that $f_j^{(m+1)}=\varphi_j\psi'_j$. Plugging this into 
\eqref{eqn:lobok} gives 
$$\mathcal W=\sum_{j=0}^{m+1}(\varphi_j\Delta_j)\cdot\psi'_j.$$ 
Here, each summand on the right is in $\mathcal N\cdot\bmoa'$, whence we infer (using Lemma \ref{lem:nevbmo} again) 
that $\mathcal W\in\mathcal N\cdot\bmoa'$. 
\par Consequently, we have 
\begin{equation}\label{eqn:phipsi}
\mathcal W=\Phi\Psi'
\end{equation}
with some $\Phi\in\mathcal N$ and $\Psi\in\bmoa$; moreover, we take $\Phi$ to be zero-free. (To see that this is 
always possible, assume that \eqref{eqn:phipsi} holds with $\Phi=\Phi_0B$, where $\Phi_0$ is zero-free and $B$ is 
a Blaschke product. Then invoke \eqref{eqn:idbmoa} to find a function $\Psi_0\in\bmoa$ such that $\Psi'_0=B\Psi'$, 
and use the factorization $\mathcal W=\Phi_0\Psi'_0$.) It now follows that the zero set $\mathcal Z(\mathcal W)$ 
coincides with $\mathcal Z(\Psi')$ and is, therefore, contained in $\mathcal Z(\bmoa')$. Inclusion \eqref{eqn:pizda2} 
is thus established. 
\quad\qed

\end{document}